\documentclass[leqno,11pt]{amsart}
\usepackage{amsmath}
\usepackage{amsfonts}
\usepackage{amssymb}
\usepackage{amsthm}

\usepackage{mathrsfs}
\usepackage{dsfont}
\usepackage{txfonts}
\usepackage{mathabx}

\usepackage{graphicx}
\usepackage[all]{xy}
\usepackage{pgf,tikz}
\usepackage{subfigure}
\usepackage{tikz-3dplot}
\usetikzlibrary{patterns,arrows}
\tikzset{
    partial ellipse/.style args={#1:#2:#3}{
        insert path={+ (#1:#3) arc (#1:#2:#3)}
    }
}

\usepackage{hyperref}

\usepackage{color}
\usepackage[shortlabels]{enumitem}

\newtheorem{theorem}{Theorem}[section]
\newtheorem{lemma}[theorem]{Lemma}
\newtheorem{corollary}[theorem]{Corollary}

\newtheorem{proposition}[theorem]{Proposition}

\theoremstyle{definition}

\theoremstyle{remark}

\newcommand{\eps}{\varepsilon}

\title[{Enhanced profile estimates}]{Enhanced profile estimates for ovals and translators}

\author{Kyeongsu Choi, Robert Haslhofer, Or Hershkovits}

\begin{document}

\begin{abstract}
We consider the profile function of ancient ovals and of noncollapsed translators.
Recall that pioneering work of Angenent-Daskalopoulos-Sesum (JDG '19, Annals '20) gives a sharp $C^0$-estimate and a quadratic concavity estimate for the profile function of two-convex ancient ovals, which are crucial in their papers as well as a slew of subsequent papers on ancient solutions of mean curvature flow and Ricci flow. In this paper, we derive a sharp gradient estimate, which enhances their $C^0$-estimate, and a sharp Hessian estimate, which can be viewed as converse of their quadratic concavity estimate. Motivated by our forthcoming work on ancient noncollapsed flows in $\mathbb{R}^4$, we derive these estimates in the context of ancient ovals in $\mathbb{R}^3$ and noncollapsed translators in $\mathbb{R}^4$, though our methods seem to apply in other settings as well.
\end{abstract}

\maketitle

\section{Introduction}

Recall that ancient ovals are compact ancient noncollapsed mean curvature flows that are not self-similar. Existence of ancient ovals has been proved more than 20 years ago in the work of White \cite{White_nature}, where it already became clear that ancient ovals play a key role in singularity analysis. On the other hand, since ancient ovals are constructed by limiting arguments, tackling questions about uniqueness and classification took much longer. In pioneering work \cite{ADS1,ADS2}, Angenent-Daskalopoulos-Sesum proved that ancient ovals in $\mathbb{R}^3$ are unique up to scaling and rigid motion.\footnote{Their result also holds in higher dimensions under the additional assumption that the solution is uniformly two-convex.} In their first paper, they established unique asymptotics. To recall the statement, given any ancient oval $M_t \subset \mathbb{R}^3$, let us considerer the profile function $v(y,\tau)$ of the renormalized flow $\bar{M}_\tau = e^{\tau/2} M_{-e^{-\tau}}$ in suitable coordinates, namely
\begin{equation}
\bar{M}_\tau = \big\{ (y,y_2,y_3)\in\mathbb{R}^3 \, : \, |(y_2,y_3)|= v(y,\tau)\big\}.
\end{equation}
Moreover, in the tip regions one defines $Y(\cdot,\tau)$ as the inverse function of $v(\cdot,\tau)$, and sets
\begin{equation}
Z(\rho,\tau) := |\tau|^{1/2} \left( Y(|\tau|^{-1/2}\rho,\tau) -  Y(0,\tau)\right).
\end{equation}

\begin{theorem}[{unique asymptotics \cite{ADS1}}]\label{thm_ADS1}
For $\tau\to -\infty$ we have the following asymptotics:
\begin{enumerate}
\item Parabolic region: Given any $K<\infty$ and $k\geq 1$, the profile function $v$ satisfies
\begin{equation*}
v(y,\tau) = \sqrt{2}-\frac{y^2-2}{\sqrt{8}|\tau|} + o(|\tau|^{-1})\qquad \mathrm{in}\,\, C^k \,\, \mathrm{for}\,\, |y|\leq K.
\end{equation*}
\item Intermediate region: Given any $\delta>0$, the function $\bar{v}(z,\tau):=v(|\tau|^{1/2}z,\tau)$ satisfies
\begin{equation*}
\bar{v}(z,\tau)=\sqrt{2-z^2}+ o(|\tau|^{-1})\qquad \mathrm{in}\,\, C^0 \,\, \mathrm{for}\,\, |z| \leq \sqrt{2}-\delta.
\end{equation*}
\item Tip region: Given any $L<\infty$ and $k\geq 1$, the tip profile function $Z(\rho,\tau)$ satisfies
\begin{equation*}
Z(\rho,\tau) = Z^B(\rho) + o(|\tau|^{-1})\qquad \mathrm{in}\,\, C^k \,\, \mathrm{for}\,\, |\rho |\leq L,
\end{equation*}
where $Z^B$ denotes the profile function of the 2d-bowl with speed $1/\sqrt{2}$.
\end{enumerate}
\end{theorem}

In their second paper, Angenent-Daskalopoulos-Sesum upgraded their unique asymptotics to a uniqueness result. The key analytic difficultly was to deal with the collar region $\{ L|\tau|^{-1/2}\leq v \leq \theta\}$, which connects the translator behavior in the tip region and the cylindrical behavior in the intermediate region. To overcome this, via a clever maximum principle argument they proved the following crucial estimate:

\begin{theorem}[{quadratic concavity \cite{ADS2}}]\label{thm_ADS2}
For $\tau$ sufficiently negative the function $y\mapsto v(y,\tau)^2$ is concave, namely
\begin{equation}
v v_{yy} +v_y^2 \leq 0\qquad \mathrm{for}\,\, \tau\ll 0.
\end{equation}
\end{theorem}

Profile estimates similar to the ones in the above two theorems played a fundamental role in many subsequent papers on ancient solutions. In particular, for the mean curvature flow this includes the uniqueness result for cohomogeneity-one ovals by Du and the second author \cite{DH_ovals}, the work on noncollapsed translators in $\mathbb{R}^4$ by the three of us \cite{CHH_translators}, and the classification of bubble-sheet ovals in $\mathbb{R}^4$ by B. Choi, Du, Daskalopoulos, Sesum and the second author \cite{CDDHS}. For the Ricci flow, this includes the classification of $3$-dimensional compact $\kappa$-solutions by Brendle, Daskalopoulos and Sesum \cite{BDS}, which upgrades the unique asymptotics from their prior joint work with Angenent \cite{ABDS}, as well as the generalization to higher dimensions under the additional PIC2 assumption from their joint work with Naff \cite{BDNS}.\\

The purpose of the present paper is to derive new profile estimates that enhance the ones from the above theorems. Our primary motivation for this comes from our work on the classification of ancient noncollapsed flows in $\mathbb{R}^4$, a program initiated in our paper \cite{CHH_wings} and a paper by Du and the second author \cite{DH_shape}, and further developed in a subsequent paper by Du and the second author \cite{DH_norotation}, and in the two above cited papers on translators \cite{CHH_translators} and bubble-sheet ovals \cite{CDDHS}. Using the terminology form \cite{DH_shape}, the relevant rank-one examples for this classification program in $\mathbb{R}^4$ are $\mathbb{R}\times$2d-ovals and the 3-dimensional noncollapsed translators constructed by Hoffman-Ilmanen-Martin-White \cite{HIMW}, and we will thus derive our new estimates in the setting of ancient ovals in $\mathbb{R}^3$ and of noncollapsed translators in $\mathbb{R}^4$. However, it seems likely that our methods will be useful in other settings as well.\\

Let us describe our new estimates in the setting of ancient ovals in $\mathbb{R}^3$. We consider the gradient function
\begin{equation}
g:=-\frac{1}{z}(\bar{v}^2)_z,
\end{equation}
where $\bar{v}(z,\tau)=v(|\tau|^{1/2}z,\tau)$ denotes the squashed profile function as above, and prove:

\begin{theorem}[gradient estimate]\label{thm:min.gradient.improved_intro}
We have
\begin{equation}
\lim_{L\to\infty}\limsup_{\tau\to -\infty}   \sup_{\bar{v}(z,\tau)\geq L|\tau|^{-1/2}} |g(z,\tau)-2|=0.
\end{equation}
\end{theorem}    

This enhances the $C^0$-estimate from Theorem \ref{thm_ADS1} (unique asymptotics) to a sharp $C^1$-estimate, valid in an optimal region. We note that on the 2d-bowl soliton one has $g\to 4$ as $\rho\to 0$ and $g\to 2$ as $\rho\to \infty$, so the limit $L\to\infty$ is indeed necessary. As we will see, our gradient estimate is most delicate in a new region that we call the belt region. Namely, the key challenge is the region $\{\sqrt{2}-K^2|\tau|^{-1}\geq v \geq \sqrt{2}-\theta\}$, which interpolates between the $C^k$-behavior in the parabolic region and the $C^0$-behavior in the intermediate region.\\

Our second main estimate for ancient ovals in $\mathbb{R}^3$ is the following global $C^2$-estimate:

\begin{theorem}[Hessian estimate]\label{thm:hessian.improved_intro}
There exists a numerical constant $C<\infty$, such that
\begin{equation}
v v_{yy} +C\left(v_y^2+\frac{1}{|\tau|}\right) \geq 0\qquad \mathrm{for}\,\, \tau\ll 0.
\end{equation}
\end{theorem}
Our Hessian estimate can be viewed as converse of Theorem \ref{thm_ADS2} (quadratic concavity). In particular, it gives quadratic smallness for the second derivates of the profile function. Together with standard interior estimates it yields quadratic smallness of all higher derivatives as well:

\begin{corollary}[higher derivative estimates]
For every $k\geq 2$ there exists a constant $C_k<\infty$, such that
\begin{equation}
v^{k-1}\left| \partial^k_y v\right| \leq C_k \left(v_y^2+\frac{1}{|\tau|}\right)\qquad\qquad \mathrm{for}\,\, \,  v(y,\tau)\geq L|\tau|^{-1/2}, \,\, \tau\ll 0.
\end{equation}
\end{corollary}

\bigskip

Let us now describe our results for noncollapsed translators in $\mathbb{R}^4$. We consider noncollapsed translators $M\subset\mathbb{R}^4$, that are neither $\mathbb{R}\times$2d-bowl nor 3d-bowl, and we normalize without loss of generality such that they move with unit speed in $x_4$-direction and such that their $\mathrm{SO}_2$-symmetry, c.f. \cite{Zhu}, is in the $x_2x_3$-variables. Denote by $V(x,t)$ the profile function of the level sets $M\cap\{x_4= -t\}$, namely
\begin{equation}
M\cap\{x_4= -t\}=\big\{(x,x_2,x_3,-t)\in\mathbb{R}^4 \, : \, |(x_2,x_3)|=V(x,t)\big\},
\end{equation}
and consider the corresponding renormalized profile function
\begin{equation}
v(y,\tau)=e^{\tau/2}V(e^{-\tau/2}y,-e^{-\tau}).
\end{equation}
This translator profile function $v(y,\tau)$ satisfies the same asymptotics as the profile function of the 2d-ovals in Theorem \ref{thm_ADS1} (unique asymptotics). More precisely, these asymptotic hold uniformly for families of translators in $\mathbb{R}^4$ that are $\kappa$-quadratic at some time $\tau_0$ in the sense of \cite[Definition 1.4]{CHH_translators}:

\begin{theorem}[{uniform asymptotics \cite{CHH_translators}}]\label{thm_transl_asy}
For every $\eps>0$ there exist constants $\kappa>0$ and $\tau_\ast>-\infty$, such that if $M$ is $\kappa$-quadratic at some time $\tau_0\leq\tau_\ast$, then for all $\tau\leq \tau_0$ the following holds:
\begin{enumerate}
\item Parabolic region: The profile function $v$ satisfies
\begin{equation*}
\left\| v(y,\tau) - \sqrt{2}+\frac{y^2-2}{\sqrt{8}|\tau|} \right\|_{C^{\lfloor \eps^{-1}\rfloor}(B(0,\eps^{-1}))}\leq \frac{\eps}{|\tau|}.
\end{equation*}
\item Intermediate region: The function $\bar{v}(z,\tau):=v(|\tau|^{1/2}z,\tau)$ satisfies
\begin{equation*}
\left|\bar{v}(z,\tau)-\sqrt{2-z^2}\right|\leq \eps \qquad  \mathrm{for}\,\, |z| \leq \sqrt{2}-\eps.
\end{equation*}
\item Tip region: The tip profile function $Z(\rho,\tau)$ satisfies the estimate
\begin{equation*}
\left\| Z(\cdot,\tau)-Z^B(\cdot) \right\|_{C^{\lfloor \eps^{-1}\rfloor}(B(0,\eps^{-1}))}\leq \eps,
\end{equation*}
where $Z^B$ denotes the profile function of the 2d-bowl with speed $1/\sqrt{2}$.
\end{enumerate}
\end{theorem}

Moreover, let us also recall the variant of the quadratic concavity estimate for translators in $\mathbb{R}^4$:

\begin{theorem}[{almost quadratic concavity \cite{CHH_translators}}]\label{almost_quad_conc}
There exist constants $\kappa>0$ and $\tau_\ast>-\infty$, such that if $M$ is $\kappa$-quadratic at some time $\tau_0\leq\tau_\ast$, then
\begin{equation}
v v_{yy} +v_y^2 - \frac{e^{\tau}}{v^2}\leq 0\qquad \mathrm{for}\,\, \tau\leq \tau_0.
\end{equation}
\end{theorem}

In this setting of noncollapsed translators in $\mathbb{R}^4$, we prove the following gradient estimate that holds uniformly for families of $\kappa$-quadratic solutions:

\begin{theorem}[gradient estimate for translators]\label{thm:min.gradient.improved_intro_trans}
For every $\eps>0$ there exist $\kappa>0$, $L<\infty$ and $\tau_\ast>-\infty$, such that if $M$ is $\kappa$-quadratic at some time $\tau_0\leq\tau_\ast$, then 
\begin{equation}
\sup_{\bar{v}(z,\tau)\geq L|\tau|^{-1/2}} |g(z,\tau)-2| \leq \eps\qquad \mathrm{for}\,\, \tau\leq \tau_0.
\end{equation}
\end{theorem}

Next, we have the following uniform Hessian estimate:

\begin{theorem}[Hessian estimate for translators]\label{thm:hessian.improved_intro_trans}
There exist constants $\kappa>0$, $\tau_\ast>-\infty$ and $C<\infty$, such that if $M$ is $\kappa$-quadratic at some time $\tau_0\leq\tau_\ast$, then
\begin{equation}
v v_{yy} +C\left(v_y^2+\frac{1}{|\tau|}\right) \geq 0\qquad \mathrm{for}\,\, \tau\leq \tau_0.
\end{equation}
\end{theorem}

In particular, this can be viewed as converse of Theorem \ref{almost_quad_conc} (almost quadratic concavity). Together with standard interior estimates this yields the corresponding higher derivative estimates as well:

\begin{corollary}[higher derivative estimates for translators]
There exist constants $\kappa>0$, $\tau_\ast>-\infty$, $L<\infty$ and $C_k<\infty$ for $k\geq 2$, such that if $M$ is $\kappa$-quadratic at some time $\tau_0\leq\tau_\ast$, then
\begin{equation}
v^{k-1}\left| \partial^k_y v\right| \leq C_k \left(v_y^2+\frac{1}{|\tau|}\right)\qquad\qquad \mathrm{for}\,\, \,  v(y,\tau)\geq L|\tau|^{-1/2}, \,\, \tau\leq \tau_0.
\end{equation}
\end{corollary}

\bigskip

Let us now outline the main ideas of the proofs. To begin with, recall from \cite{ADS1} that the squashed profile function of ancient ovals in $\mathbb{R}^3$ evolves by
\begin{equation}\label{eq:bar_v.evol_intro}
\bar v_\tau= \frac{\bar v_{zz}}{|\tau|+\bar v_z^2}-\frac{z}{2}\left(1+\frac{1}{|\tau|}\right)\bar v_z+\frac{\bar v}{2}-\frac{1}{\bar v}. 
\end{equation}
In particular, observe that the ellipticity degenerates as $\tau\to -\infty$, which is of course expected in light of the squashing. Hence, all standard interior estimates inevitably loose some $|\tau|$-factors, and the only hope to derive sharp estimates is to cook up suitable quantities that satisfy a nice maximum principle.

Regarding the gradient estimate, we first observe that the statement away from the belt region follows quite easily from Theorem \ref{thm_ADS1} (unique asymptotics) and convexity. To deal with the belt region $\{\sqrt{2}-K^2|\tau|^{-1}\geq v \geq \sqrt{2}-\theta\}$ we then apply the maximum principle to the functions
\begin{equation}
g^\pm :=(1\pm z)g.
\end{equation}

Regarding the Hessian estimate, we first consider the quantity
\begin{equation}
h^\ast:=(2-|z|)(-q^2q_{zz}+qq_z^2),
\end{equation}
and use our gradient estimate to control various term in the resulting evolution equation. Applying the maximum principle we then derive the intermediary Hessian estimate
\begin{equation}
\bar{v}_{zz}\geq -\frac{1000}{\bar{v}^5}.
\end{equation}
This estimate is pretty good in the intermediate region, but far from optimal in the collar region. To deal with the collar region, working with the unrescaled profile function $V(x,t)$ for convenience, we then consider the quantity
\begin{equation}
H=(V^p)_{xx}.
\end{equation}
Choosing the exponent $p$ sufficiently large we show that $H$ is a supersolution of a suitable parabolic equation (this is in contrast to \cite{ADS2} who got a subsolution for $p=2$). Using our prior estimates to control the sign at the boundary, we can then conclude the proof of the Hessian estimate.\\

Finally, the case of noncollapsed translators involves two modifications of the above. First, we have to control some additional nonlinear error terms in the evolution of $\bar{v}$. Second, we have to check that all estimates are uniform for families of $\kappa$-quadratic solutions. This concludes the outline of the proof.\\

\bigskip

\noindent\textbf{Acknowledgments.}
KC has been supported by KIAS Individual Grant MG078901 and a POSCO Science Fellowship. RH has been supported by an NSERC Discovery Grant and a Sloan Research Fellowship. OH has been supported by ISF grant 437/20.

\bigskip

\section{Gradient estimate}

Denoting by $v(y,\tau)$ the profile function of $\bar{M}_\tau$, we work with the squashed profile function 
\begin{equation}
\bar{v}(z,\tau):=v(|\tau|^{1/2}z,\tau).
\end{equation}
We consider the function
\begin{equation}
g:=-\frac{1}{z}q_z,\qquad \mathrm{where}\qquad q=\bar{v}^2.
\end{equation}
Let us first observe that the function $g$ indeed has the expected behaviour away from some small regions:

\begin{proposition}[gradient in good regions]\label{lem:gradient.conv.tale} For any constant $\theta>0$ we have
\begin{equation}\label{grad_good1}
\lim_{L\to\infty}\limsup_{\tau\to -\infty}   \sup_{\bar{v}(z,\tau)\in [L|\tau|^{-1/2},\sqrt{2}-\theta]} |g(z,\tau)-2|=0.
\end{equation}
Moreover, for any constant $K<\infty$ we have
\begin{equation}\label{eq:zq_z.tale}
\lim_{\tau\to -\infty}   \sup_{ z\in  [0,K|\tau|^{-1/2}]} |g(z,\tau)-2|=0.
\end{equation}
\end{proposition}

\begin{proof}In the following we will frequently use Theorem \ref{thm_ADS1} (unique asymptotics). To begin with,
recall that for $\tau$ sufficiently negative the function $q$ is concave thanks to Theorem \ref{thm_ADS2} (quadratic concavity). Hence, the locally uniform $C^0$-convergence of $q(\cdot,\tau)$ to $2-z^2$ from the intermediate region asymptotics can be upgraded to locally uniform $C^1$-convergence, namely for any $\delta>0$ we get
\begin{equation}
\lim_{\tau\to -\infty}\sup_{|z|\leq \sqrt{2}-\delta}|q_z(z,\tau)+2z|=0.
\end{equation}
This yields
\begin{equation}\label{eq:zq_z.middle}
\lim_{\tau\to -\infty}   \sup_{ \delta \leq |z|\leq \sqrt{2}-\delta} |g(z,\tau)-2|=0.
\end{equation}
Next, to deal with the region where $|z|\geq \sqrt{2}-\delta$ and $\bar{v}(z,\tau)\geq L|\tau|^{-1/2}$, observe that
\begin{equation}
q_z(z,\tau)=\frac{2\rho}{Z_\rho(\rho)} \qquad \mathrm{where}\qquad \rho = |\tau|^{1/2}\bar{v}(z,\tau).
\end{equation}
Recall also that the profile function of the 2d-bowl with speed $1/\sqrt{2}$ satisfies
\begin{equation}
Z^B_\rho(\rho)=-\frac{1}{\sqrt{2}}\rho+O( \rho^{-1})\qquad \mathrm{for}\quad \rho\to\infty.
\end{equation}
By the tip region asymptotics, remembering also the concavity of $q$, this yields
\begin{equation}
\liminf_{L\to\infty}\liminf_{\tau\to -\infty}\inf_{\bar{v}(z,\tau)\geq L|\tau|^{-1/2} } q_z \geq -2\sqrt{2}.
\end{equation}
Combining with \eqref{eq:zq_z.middle}, and remembering also the diameter asymptotics that follow from the intermediate region asymptotics, this establishes \eqref{grad_good1}.\\
Finally, to deal with the parabolic region recall that for $|y|\leq K$ we have
\begin{equation}\label{asympt_parab}
v(y,\tau)=\sqrt{2}-\frac{y^2-2}{\sqrt{8}|\tau|}+o(|\tau|^{-1})    
\end{equation}
in the $C^k$-sense. Taking also into account that $v_y(0,\tau)=0$ thanks to the $\mathbb{Z}_2$-symmetry it follows that
\begin{equation}
   \frac{v_y}{y} =- \frac{1}{\sqrt{2}|\tau|}+o(|\tau|^{-1}). 
\end{equation}
For $|y|\leq K$ we thus conclude that
\begin{equation}
    g(|\tau|^{-1/2}y,\tau)=-\frac{2|\tau| v(y,\tau)  v_y(y,\tau)}{y}  = 2+o(1).
\end{equation}
This shows \eqref{eq:zq_z.tale}, and thus completes the proof.
\end{proof}

To deal with the belt region, namely the region where $ \sqrt{2}-K^2|\tau|^{-1}\geq v \geq \sqrt{2}-\theta$, let us begin by computing the relevant evolution equation:

\begin{proposition}[evolution of gradient function]\label{prop:evol.pseudo_hessian}
The function $g=-\frac{1}{z}q_z$ evolves by
\begin{multline}\label{eq:evol.pseudo_hessian}
    g_\tau=\frac{4qg_{zz}}{4q|\tau|+q_z^2}+\frac{8qg_{z}}{z(4q|\tau|+q_z^2)}-\frac{z}{2}\left(1+\frac{1}{|\tau|}\right)g_z -\frac{g}{|\tau|}\\
    +\frac{4\left(2q(g+zg_z)+q_z^2\right)\left(2|\tau|-(g+zg_z)\right)}{(4q|\tau|+q_z^2)^2}g.
\end{multline}
\end{proposition}

\begin{proof}Recall from \cite[Section 6]{ADS1} that the squashed profile function evolves by
\begin{equation}\label{eq:bar_v.evol}
\bar v_\tau= \frac{\bar v_{zz}}{|\tau|+\bar v_z^2}-\frac{z}{2}\left(1+\frac{1}{|\tau|}\right)\bar v_z+\frac{\bar v}{2}-\frac{1}{\bar v}. 
\end{equation}
Multiplying this by $2 \bar v$ yields
\begin{equation}
(\bar v^2)_\tau= \frac{(\bar v^2)_{zz}-2\bar v_z^2}{|\tau|+\bar v_z^2} -\frac{z}{2} \left(1+\frac{1}{|\tau|}\right)(\bar v^2)_z+ \bar v^2-2,
\end{equation}  
which can be rewritten in the form
\begin{equation}\label{eq:q.evol}
q_\tau= \frac{4qq_{zz}-2q_z^2}{4q|\tau|+q_z^2} -\frac{z}{2} \left(1+\frac{1}{|\tau|}\right)q_z+ q-2.
\end{equation}
Differentiating this with respect to $z$ implies
\begin{equation}\label{eq:q_z.evol}
q_{z\tau}= \frac{4qq_{zzz}}{4q|\tau|+q_z^2} -\frac{z}{2} \left(1+\frac{1}{|\tau|}\right)q_{zz}+\frac{1}{2} \left(1-\frac{1}{|\tau|}\right)q_z-\frac{(4qq_{zz}-2q_z^2)(4|\tau|+2q_{zz})}{(4q|\tau|+q_z^2)^2}q_z. 
\end{equation}
Finally, dividing this by $-z$ yields
\begin{equation}
g_\tau=-\frac{4z^{-1}qq_{zzz}}{4q|\tau|+q_z^2} +\frac{1}{2} \left(1+\frac{1}{|\tau|}\right)q_{zz}+\frac{1}{2} \left(1-\frac{1}{|\tau|}\right)g-\frac{4(2qq_{zz}-q_z^2)(2|\tau|+q_{zz})}{(4q|\tau|+q_z^2)^2}g. 
\end{equation}
Observing also that $q_{zz}=-(zg_z+g)$ and $q_{zzz}=-(zg_{zz}+2g_z)$, this implies the assertion.
 \end{proof}

We now consider the functions
\begin{equation}
g^\pm:=(1 \pm z)g.
\end{equation}

\begin{lemma}[differential inequalities for gradient functions]\label{lem:min.gradient}
Given any sufficiently small $\delta>0$, there exists $\tau_\ast> -\infty$ with the following significance:
\begin{enumerate}[(i)]
\item If $g^-$ attains its local maximum at $(z_0,\tau_0)$ satisfying $\tau_0\leq \tau_\ast$ and $5|\tau_0|^{-\frac{1}{2}}\leq z_0 \leq \delta$, then  
\begin{equation}
    g^-_\tau(z_0,\tau_0) \leq - |\tau_0|^{-\frac{1}{2}}g^- (z_0,\tau_0).
\end{equation}\label{diff_i}
\item If $g^+$ attains its local minimum at $(z_0,\tau_0)$ satisfying $\tau_0\leq \tau_\ast$ and $5|\tau_0|^{-\frac{1}{2}}\leq z_0 \leq \delta$, then 
\begin{equation}
    g^+_\tau (z_0,\tau_0) \geq  |\tau_0|^{-\frac{1}{2}}g^+(z_0,\tau_0).
    \end{equation}\label{diff_ii}
\end{enumerate}
\end{lemma}

\begin{proof}
Observe first that at any local maximum of $ \mp g^\pm(\cdot,\tau)$ we have
\begin{align}\label{eq:h.pm.critical}
&\mp (1 \pm z)g_z=  g, &&  \mp (1 \pm z)g_{zz}\leq \mp  2(1\pm z)^{-1}g.
\end{align}
We now multiply the evolution equation from Proposition \ref{prop:evol.pseudo_hessian} (evolution of gradient function) by $\mp( 1\pm z)$. Thanks to \eqref{eq:h.pm.critical} the resulting first two terms satisfy
\begin{align}
\mp\frac{4q(1 \pm z)g_{zz}}{4q|\tau|+q_z^2}\mp\frac{8q(1 \pm z)g_z}{z(4q|\tau|+q_z^2)}&\leq  \frac{8qg }{z(1 \pm z) (4q|\tau|+q_z^2)}
\end{align}
at any local maximum of $\mp g^\pm(\cdot,\tau)$ in $\{5|\tau|^{-1/2}\leq z\leq \delta\}$. Observing also that $g+zg_z=(1\pm z)^{-2}g^\pm$ again thanks to \eqref{eq:h.pm.critical}, at any such local maximum we thus get
\begin{multline}\label{loc_min_ineq}
   \mp g^\pm_\tau\leq \frac{8qg^\pm }{z(1 \pm z)^2 (4q|\tau|+q_z^2)}-\frac{z}{2(1\pm z)}\left(1+\frac{1}{|\tau|}\right)g^\pm \pm \frac{g^\pm}{|\tau|}\\
    \qquad \mp\frac{4(2qg^\pm+(1 \pm  z)^2 q_z^2)(2(1 \pm  z)^2|\tau|-g^\pm)}{(1 \pm  z)^4(4q|\tau|+q_z^2)^2}g^\pm.
\end{multline}
To proceed, observe that thanks to Proposition \ref{lem:gradient.conv.tale} (gradient in good regions) and convexity for $5|\tau|^{-1/2}\leq z\leq \delta$, with $\delta$ sufficiently small and $\tau$ sufficiently negative, we have $|q_z|\leq 3\delta$, and thus
\begin{equation}\label{growthbound_gpm}
|g^\pm| \leq \delta |\tau|^{1/2}.
\end{equation}
Now, at any local maximum of $g^-$ in $\{5|\tau|^{-1/2}\leq z\leq \delta\}$ for $\tau$ sufficiently negative we can estimate
\begin{align}
    \frac{8qg^- }{z(1 - z)^2 (4q|\tau|+q_z^2)}-\frac{z}{2(1- z)}\left(1+\frac{1}{|\tau|}\right)g^- - \frac{g^-}{|\tau|}\leq \frac{9g^- }{4z|\tau|}-\frac{z}{2}g^-\leq -\frac{2g^-}{|\tau|^{1/2}},
\end{align}
and
\begin{equation}
\frac{4(2qg^-+(1 - z)^2 q_z^2)(2(1 - z)^2|\tau|-g^-)}{(1 -  z)^4(4q|\tau|+q_z^2)^2}g^-\leq \frac{10(\delta |\tau|^{1/2}+\delta^2)|\tau|}{|\tau|^2}g^- \leq \frac{g^- }{ |\tau|^{1/2}},
\end{equation}
which proves \ref{diff_i}.\\
On the other hand, at any local minimum of $g^+$ in $\{5|\tau|^{-1/2}\leq z\leq \delta\}$ for $\tau$ sufficiently negative the term in the second line of
\eqref{loc_min_ineq} has the good sign, and we can thus estimate
\begin{align}
   - g^+_\tau\leq  \frac{2g^+ }{z |\tau|}-\frac{z}{2(1+ z)}g^+ + \frac{g^+}{|\tau|}\leq -\frac{g^+}{|\tau|^{1/2}},
\end{align}
which proves \ref{diff_ii}.
\end{proof}

We can now prove Theorem \ref{thm:min.gradient.improved_intro} (gradient estimate), which we restate for convenience of the reader:

\begin{theorem}[gradient estimate]\label{thm:min.gradient.improved}
For ancient ovals in $\mathbb{R}^3$ we have
\begin{equation}
\lim_{L\to\infty}\limsup_{\tau\to -\infty}   \sup_{\bar{v}(z,\tau)\geq L|\tau|^{-1/2}} |g(z,\tau)-2|=0.
\end{equation}
\end{theorem}    

\begin{proof}By Proposition \ref{lem:gradient.conv.tale} (gradient in good regions) it is enough to show that for any $\delta>0$ the functions
\begin{equation}
\overline{g}(\tau):=\sup_{z \in [5|\tau|^{-1/2},\delta]}  g^-(z,\tau) \quad \mathrm{and}\quad \underline{g}(\tau):=\inf_{z \in [5|\tau|^{-1/2},\delta]}  g^+(z,\tau)
\end{equation}
satisfy
\begin{equation}\label{thm_grad_to_show}
    \limsup_{\tau \to -\infty}\overline{g}(\tau)\leq 2
\quad \mathrm{and} \quad
    \liminf_  {\tau \to -\infty}\underline{g}(\tau)\geq 2.
\end{equation}
Also note that, given any $\eps>0$, again by Proposition \ref{lem:gradient.conv.tale} (gradient in good regions) there exists some $\tau_{\delta,\varepsilon}>-\infty$ such that for $\tau\leq \tau_{\delta,\varepsilon}$ we have
\begin{equation}
 \sup_{z \in [0,5|\tau|^{-1/2})\cup (\delta,2\delta]} g^-(z,\tau)\leq 2+\varepsilon \quad \mathrm{and}\quad   \inf_{z \in [0,5|\tau|^{-1/2})\cup (\delta,2\delta]} g^+(z,\tau)\geq 2-\varepsilon.
\end{equation}

Now, suppose towards a contradiction that there is some $\tau_0\leq \tau_{\delta,\varepsilon}$ such that $\overline{g}(\tau_0)\geq 2+2\eps$. 
Then, applying Lemma \ref{lem:min.gradient} (differential inequalities for gradient functions) we infer that the function
\begin{equation}
    f(\tau):=\log \overline{g}(\tau)-|\tau|^{1/2}
\end{equation}
for all $\tau\leq\tau_0$ satisfies\footnote{In particular, for all $\tau\leq\tau_0$ the supremum of $g^{-}(\cdot,\tau)$ over $ [0,2\delta]$ is attained in $[5|\tau|^{-1/2},\delta]$.}
\begin{equation}
-f_{\tau}\geq 0.
\end{equation}
For all $\tau\leq\tau_0$ this implies
\begin{equation}
   \log\overline{g}(\tau) - |\tau|^{1/2}\geq f(\tau_0),
\end{equation}
contradicting the growth bound for $g^-$ from \eqref{growthbound_gpm}. Since $\eps>0$ was arbitrary, this proves the upper bound for $\overline{g}$.

Similarly, suppose towards a contraction that there is some $\tau_0\leq \tau_{\delta,\varepsilon}$ such that $\underline{g}(\tau_0)\leq 2-2\eps$.
Then, applying Lemma \ref{lem:min.gradient} (differential inequalities for gradient functions) for $\tau\leq\tau_0$ we infer that
\begin{equation}
   \log\underline{g}(\tau) + |\tau|^{1/2}\leq C,
\end{equation}
where $C=C(\tau_0)<\infty$. Together with the concavity of $q$ from Theorem \ref{thm_ADS2} (quadratic concavity), for $\tau\leq\tau_0$ this implies
\begin{equation}
|q_z(|\tau|^{-\frac{1}{2}},\tau)|\leq Ce^{-|\tau|^{\frac{1}{2}}},
\end{equation}
contradicting the asymptotics in the parabolic region from \eqref{asympt_parab}. Since $\eps>0$ was arbitrary, this proves the lower bound for $\underline{g}$, and thus concludes the proof of the theorem.
\end{proof}

Finally, let us describe the needed modifications for noncollapsed translators in $\mathbb{R}^4$:

\begin{proof}[{Proof of Theorem \ref{thm:min.gradient.improved_intro_trans} (gradient estimate for translators)}] Arguing as in the proof of Proposition \ref{lem:gradient.conv.tale} (gradient in good regions), but now using Theorem \ref{thm_transl_asy} (uniform asymptotics) and Theorem \ref{almost_quad_conc} (almost quadratic concavity) in lieu of Theorem \ref{thm_ADS1} (unique asymptotics) and Theorem \ref{thm_ADS2} (quadratic concavity), given any $\eps>0$, $\theta>0$ and $K<\infty$ we can find $\kappa>0$, $L<\infty$ and $\tau_\ast>-\infty$ with the following significance. If $M$ is $\kappa$-quadratic at some time $\tau_0\leq\tau_\ast$, then for all $\tau\leq\tau_0$ we have
\begin{equation}\label{grad_good_transl}
\sup_{\bar{v}(z,\tau)\in [L|\tau|^{-1/2},\sqrt{2}-\theta]} |g(z,\tau)-2|+ \sup_{ z\in  [0,K|\tau|^{-1/2}]} |g(z,\tau)-2|\leq \eps.
\end{equation}
Next, recall from \cite[Proposition 5.3]{CHH_translators} that the profile function of our translators evolves by
\begin{equation}
v_\tau= \frac{v_{yy}}{1+v_y^2}-\frac{y}{2} v_y+\frac{v}{2}-\frac{1}{\bar v}+e^\tau N[v], 
\end{equation}
where the nonlinear term $N[v]$ and its derivatives are well controlled thanks to the derivative estimates from \cite[Lemma 5.6]{CHH_translators}. Using this we infer that
\begin{multline}\label{eq:evol.pseudo_hessian_transl}
    g_\tau=\frac{4qg_{zz}}{4q|\tau|+q_z^2}+\frac{8qg_{z}}{z(4q|\tau|+q_z^2)}-\frac{z}{2}\left(1+\frac{1}{|\tau|}\right)g_z -\frac{g}{|\tau|}\\
    +\frac{4\left(2q(g+zg_z)+q_z^2\right)\left(2|\tau|-(g+zg_z)\right)}{(4q|\tau|+q_z^2)^2}g+E,
\end{multline}
where the only difference with Proposition \ref{prop:evol.pseudo_hessian} (evolution of gradient function) is the error term $E$, which, assuming $\kappa$-quadraticity at time $\tau_0\leq\tau_\ast$, in the belt region $\{ \sqrt{2}-\theta\leq \bar{v}\leq\sqrt{2}-K^2|\tau|^{-1}\}$ satisfies
\begin{equation}
|E|\leq e^{\frac{99}{100}\tau}\qquad \mathrm{for}\,\, \tau\leq\tau_0.
\end{equation}
Also observe that in the belt region, thanks to the asymptotics in the parabolic region and convexity, we have the coarse lower bound
\begin{equation}
g\geq |\tau|^{-1}\qquad \mathrm{for}\,\, \tau\leq\tau_0.
\end{equation}
Consequently, if $\mp g^\pm$ attains a local maximum in the belt region at some $(\bar{x}_0,\bar{\tau}_0)$, where $\bar{\tau}_0\leq\tau_0\leq\tau_\ast$, then assuming $\kappa$-quadraticity at time $\tau_0$ we still get the differential inequalities
\begin{equation}
\mp g^{\pm}_\tau(\bar{x}_0,\bar{\tau}_0) \leq -\frac{1}{2}|\bar{\tau}_0|^{-1/2}g^{\pm}(\bar{x}_0,\bar{\tau}_0).
\end{equation}
Utilizing these differential inequalities as above, the assertion of the theorem follows.
\end{proof}

\bigskip

\section{Hessian estimate in intermediate region}

In this section, we consider the function
\begin{equation}
h:=-q^2q_{zz}+qq_z^2.
\end{equation}

\begin{proposition}[{asymptotics of $h$}]\label{lem:rough.derivative.esti}
We have
\begin{equation}
    \lim_{\tau\to -\infty}\sup_{q(z,\tau)\geq 2}  |h(z,\tau)-8|=0\qquad\mathrm{and}\qquad
    \lim_{\tau\to -\infty}\sup_{ q(z,\tau)\leq |\tau|^{-2/3}}h(z,\tau)=0.
\end{equation}
\end{proposition}

\begin{proof}
For $|y|\leq 3$ thanks to the asymptotics in the parabolic region from \eqref{asympt_parab} we have
\begin{equation}
v_y(y,\tau)= -\frac{y}{\sqrt{2}|\tau|}+o(|\tau|^{-1})\qquad\textrm{and}\qquad v_{yy}(y,\tau)= -\frac{1}{\sqrt{2}|\tau|}+o(|\tau|^{-1}).
\end{equation}
Using this, for $|y|\leq 3$ we see that
\begin{equation}
h(|\tau|^{-1/2}y,\tau)=2|\tau| v(y,\tau)^4\left(-v(y,\tau) v_{yy}(y,\tau)+2v_y(y,\tau)^2\right)=8+o(1),
\end{equation}
which, observing also that $\bar{v}(|\tau|^{-1/2}y,\tau)< \sqrt{2}$ for $|y|\geq 3$, implies the first assertion.

Next, note that by standard cylindrical estimates \cite[Proposition 2.15]{ADS2}, given any $\eps>0$, for $L<\infty$ sufficiently large and $\tau$ sufficiently negative in the region $\bar{v}(z,\tau)\geq L|\tau|^{-1/2}$ we have
\begin{align}\label{std_cyl_est}
     q_z^2 +q|q_{zz}|\leq \eps |\tau|q.
\end{align}
On the other hand, observe that
\begin{equation}
q_z(z,\tau)=\frac{2\rho}{Z_\rho(\rho)} \qquad\textrm{and}\qquad  q_{zz}(z,\tau)=\frac{2(Z_\rho(\rho)-\rho Z_{\rho\rho}(\rho))}{Z_\rho(\rho)^3}|\tau| \qquad \mathrm{where}\quad \rho = |\tau|^{1/2}\bar{v}(z,\tau).
\end{equation}
Recall also that the profile function of the 2d-bowl with speed $1/\sqrt{2}$ for $\rho\to 0$ satisfies
\begin{equation}
Z^B_\rho(\rho)=-\frac{1}{2\sqrt{2}}\rho+O( \rho^{-1})\qquad\textrm{and}\qquad  Z^B_{\rho\rho}(\rho)=-\frac{1}{2\sqrt{2}}+O( \rho^{-2}).
\end{equation}
Together with the tip asymptotics from Theorem \ref{thm_ADS1} (unique asymptotics) for $\bar{v}(z,\tau)\leq L|\tau|^{-1/2}$ this yields
\begin{align}
     q_z^2 +q|q_{zz}|\leq C.
\end{align}
Combining these two estimates, the second assertion follows.
\end{proof}

 \begin{proposition}[{evolution of $h$}]\label{prop:w.evol}
The function $h:=-q^2q_{zz}+qq_z^2$ satisfies  
 \begin{multline}
h_\tau= \frac{4qh_{zz}-20 q_zh_z}{4q|\tau|+q_z^2}+\frac{20qq_z^2h-8h^2+6q^2q_z^4}{q^2(4q|\tau|+q_z^2)}-\frac{z}{2} \left(1+\frac{1}{|\tau|}\right)h_z+\left(2-\frac{4}{q}-\frac{1}{|\tau|}\right)h+2q_z^2\\
+\frac{4(2h-qq_z^2)}{q(4q|\tau|+q_z^2)^2}\left[3q_zh_z-3q_z^4+(h+qq_z^2)(2|\tau|+q_{zz})+\frac{4q^2q_z^2(2|\tau|+q_{zz})^2}{4q|\tau|+q_z^2}\right].
\end{multline}
\end{proposition}

\begin{proof}
Recall from the proof of Proposition \ref{prop:evol.pseudo_hessian} (evolution of gradient function) that
\begin{equation}\label{eq:q.evol_res}
q_\tau= \frac{4qq_{zz}-2q_z^2}{4q|\tau|+q_z^2} -\frac{z}{2} \left(1+\frac{1}{|\tau|}\right)q_z+ q-2,
\end{equation}
and
\begin{equation}\label{eq:q_z.evol_res}
q_{z\tau}= \frac{4qq_{zzz}}{4q|\tau|+q_z^2} -\frac{z}{2} \left(1+\frac{1}{|\tau|}\right)q_{zz}+\frac{1}{2} \left(1-\frac{1}{|\tau|}\right)q_z+\frac{4(q_z^2-2qq_{zz})(2|\tau|+q_{zz})}{(4q|\tau|+q_z^2)^2}q_z. 
\end{equation}
Differentiating again with respect to $z$ it follows that
\begin{multline}\label{eq:q_zz.evol}
q_{zz\tau}= \frac{4qq_{zzzz}-4q_zq_{zzz}}{4q|\tau|+q_z^2}-\frac{z}{2} \left(1+\frac{1}{|\tau|}\right)q_{zzz}-\frac{q_{zz}}{|\tau|}
+\frac{4(q_z^2-2qq_{zz})(2|\tau|+q_{zz})}{(4q|\tau|+q_z^2)^2} q_{zz} \\
+\frac{12q_zq_{zzz}(q_z^2-2qq_{zz}) }{(4q|\tau|+q_z^2)^2}-\frac{16q_z^2(q_z^2-2qq_{zz})(2|\tau|+q_{zz})^2}{(4q|\tau|+q_z^2)^3} .     
\end{multline}
Now, using these equations we infer that
\begin{multline}
(qq_z^2)_{\tau}= \frac{4q(qq_z^2)_{zz}-16qq_z^2q_{zz}-8q^2q_{zz}^2-2q_z^4}{4q|\tau|+q_z^2} -\frac{z}{2} \left(1+\frac{1}{|\tau|}\right)(qq_z^2)_z+\left(1-\frac{1}{|\tau|}\right)qq_z^2\\
+ (q-2)q_z^2 +\frac{8qq_z^2(q_z^2-2qq_{zz})(2|\tau|+q_{zz})}{(4q|\tau|+q_z^2)^2}  ,
\end{multline}
and
\begin{multline}
(q^2q_{zz})_\tau=\frac{4q(q^2q_{zz})_{zz}-12qq_z^2q_{zz}-20q^2q_zq_{zzz}}{4q|\tau|+q_z^2} -\frac{z}{2} \left(1+\frac{1}{|\tau|}\right)(q^2q_{zz})_z   -\frac{q^2q_{zz}}{|\tau|}+ 2q(q-2)q_{zz}\\
 +\frac{4(q_z^2-2qq_{zz})}{(4q|\tau|+q_z^2)^2}\left[3q^2q_zq_{zzz}+q^2q_{zz}(2|\tau|+q_{zz})-\frac{4q^2q_z^2(2|\tau|+q_{zz})^2}{4q|\tau|+q_z^2}\right].
\end{multline}
Subtracting these two equations the assertion follows.
\end{proof}

We now consider the function
\begin{equation}
h^\ast:=(2-|z|)h.
\end{equation}

\begin{lemma}[{differential inequality for $h^\ast$}]\label{lem:Hessian.max.speed}
If $\sup h^\ast(\cdot,\tau_0)\geq 1000$ at some $\tau_0\ll 0$, then $h^\ast(\cdot,\tau_0)$ attains its maximum at an interior point $z_0$ satisfying
$|\tau_0|^{-2/3} < q(z_0,\tau_0)< 2$ and we have
\begin{equation}
    h^\ast_\tau(z_0,\tau_0) \leq -\frac{h^\ast(z_0,\tau_0)^2}{150 |\tau_0|}.
\end{equation}
\end{lemma}

\begin{proof}
If $\sup h^\ast(\cdot,\tau_0)\geq 1000$ at some $\tau_0\ll 0$, then by Proposition \ref{lem:rough.derivative.esti} (asymptotics of $h$) the function $h^\ast(\cdot,\tau_0)$ attains its maximum at a point $z_0>0$, such that
\begin{equation}\label{comp_at_max}
h^\ast\geq 1000\qquad\textrm{and}\qquad
|\tau|^{-2/3} < q< 2.
\end{equation}
holds at $(z_0,\tau_0)$. Computing from now on at this point, observe that by Theorem \ref{thm:min.gradient.improved} (sharp gradient estimate) and the second derivative estimate from \eqref{std_cyl_est} we have
\begin{equation}\label{eq:qq_z^2.assumption}
q_z^2\leq 5z^2,\qquad qq_z^2 \leq 5\qquad\mathrm{and}\qquad 
-q_{zz}\leq \eps |\tau|.
\end{equation}
Using the latter two inequalities and the assumption that $h^\ast\geq 1000$ we can estimate
\begin{equation}\label{2nd_useful_est}
6q^2q_z^4 \leq \frac{3}{50}h qq_z^2\qquad\mathrm{and}\qquad 
\frac{4q^2q_z^2(2|\tau|+q_{zz})^2}{4q|\tau|+q_z^2}\leq (2+\eps)^2qq_z^2|\tau|.
\end{equation}
By Proposition \ref{prop:w.evol} (evolution of $h$), observing also that some terms have the good sign, specifically
\begin{equation}
\left(2-\frac{4}{q}-\frac{1}{|\tau|}\right)h\leq 0\qquad\mathrm{and}\qquad  \frac{4(2h-qq_z^2)}{q(4q|\tau|+q_z^2)^2}\left[-3q_z^4+(h+qq_z^2)q_{zz}\right]\leq 0,
\end{equation}
we have the evolution inequality
 \begin{multline}
h_\tau\leq \frac{4qh_{zz}-20 q_zh_z}{4q|\tau|+q_z^2}+\frac{20qq_z^2h-8h^2+6q^2q_z^4}{q^2(4q|\tau|+q_z^2)}-\frac{z}{2}\left[ \left(1+\frac{1}{|\tau|}\right)h_z-\frac{4}{z}q_z^2\right]\\
+\frac{12(2h-qq_z^2)q_zh_z}{q(4q|\tau|+q_z^2)^2}+\frac{4(2h-qq_z^2)}{q(4q|\tau|+q_z^2)^2}\left[2|\tau|(h+qq_z^2)+\frac{4q^2q_z^2(2|\tau|+q_{zz})^2}{4q|\tau|+q_z^2}\right].
\end{multline}
Using \eqref{2nd_useful_est} we can estimate the last term by
\begin{multline}
\frac{4(2h-qq_z^2)}{q(4q|\tau|+q_z^2)^2}\left[2|\tau|(h+qq_z^2)+\frac{4q^2q_z^2(2|\tau|+q_{zz})^2}{4q|\tau|+q_z^2}\right]\\
\leq \frac{4h^2}{q^2(4q|\tau|+q_z^2)}+ \frac{\left(2+2(2+\eps)^2\right)qq_z^2h}{q^2(4q|\tau|+q_z^2)}.
\end{multline}
Using \eqref{eq:qq_z^2.assumption} and again \eqref{2nd_useful_est}  we thus infer that
\begin{equation}\label{eq:w.evol}
h_\tau\leq  \frac{4qh_{zz}-20 q_zh_z }{4q|\tau|+q_z^2} +\frac{12(2h-qq_z^2)q_zh_z}{q(4q|\tau|+q_z^2)^2}+ \frac{31qq_z^2 h-4h^2}{q^2(4q|\tau|+q_z^2)}-\frac{z}{2}\left[ \left(1+\frac{1}{|\tau|}\right)h_z
-20z\right].
\end{equation}
Now, since we are at a maximum of $h^\ast$, we have
\begin{align}\label{eq:phi.max.ineq}
h_z=\frac{h^\ast}{(2-z)^2}  \qquad\mathrm{and}\qquad   h_{zz}\leq \frac{2h^\ast}{(2-z)^3}.
\end{align}
Multiplying  \eqref{eq:w.evol} by $2-z$, and considering also the signs implied by \eqref{comp_at_max} and \eqref{eq:qq_z^2.assumption} and the concavity of $q$, this yields
 \begin{equation}
 h^\ast_\tau\leq  \frac{8q^3h^\ast-20 q^2q_z(2-z)h^\ast+31qq_z^2 (2-z)^2 h^\ast-4(2-z)h^{\ast 2} }{(2-z)^2q^2(4q|\tau|+q_z^2)}.
\end{equation}
Finally, using again \eqref{comp_at_max} and \eqref{eq:qq_z^2.assumption} and the fact that $2-z \geq 2-\sqrt{2}-o(1)$ we see that
\begin{equation}
8q^3h^\ast-20 q^2q_z(2-z)h^\ast+31qq_z^2 (2-z)^2 h^\ast-4(2-z)h^{\ast 2}\leq - h^{\ast 2}
\end{equation}
and
\begin{equation}
(2-z)^2q^2(4q|\tau|+q_z^2)\leq 150 |\tau|.
\end{equation}
This implies the assertion.
\end{proof}

\begin{theorem}[Hessian estimate in intermediate region]\label{thm:Hessian.cyl.esti}
For all $\tau\ll 0$ we have
\begin{equation}
 \bar v_{zz}\geq -\frac{1000}{\bar{v}^5}.
\end{equation}
\end{theorem}

\begin{proof}
Setting $\overline{h}(\tau):=\sup_{z} h^\ast(z,\tau)$, we consider the function
\begin{equation}
f(\tau):=150 \overline{h}(\tau)^{-1}+\log|\tau|.
\end{equation}
If we had $\overline{h}(\tau_0)\geq 1000$ at some $\tau_0\ll 0$, then by Lemma \ref{lem:Hessian.max.speed} (differential inequality for $h^\ast$) for all $\tau\leq \tau_0$ we would obtain $-f_{\tau} \leq 0$, contradicting the fact that $\lim_{\tau\to -\infty} f(\tau)=\infty$. We have thus shown that that there exists a $\tau_\ast>-\infty$ such that
\begin{equation}
\sup_{\tau\leq\tau_\ast}\overline{h}(\tau)\leq 1000.
\end{equation}
Remembering the fact that $2-|z| \geq 1/2$, we thus conclude that
\begin{equation}
    2\bar v^5\bar v_{zz} \geq 2\bar v^5\bar v_{zz}-2\bar v^4\bar v_z^2 =q^2q_{zz}-qq_z^2\geq  -2000
\end{equation}
for all $\tau\leq\tau_\ast$. This proves the theorem.
\end{proof}

Finally, let us establish the corresponding variant for noncollapsed translators in $\mathbb{R}^4$:

\begin{theorem}[Hessian estimate in intermediate region for translators]\label{thm:Hessian.cyl.esti_trans}
There exist $\kappa>0$ and $\tau_\ast>-\infty$ with the following significance. If $M$ is $\kappa$-quadratic at some time $\tau_0\leq\tau_\ast$, then for $\tau\leq\tau_0$ we have
\begin{equation}
 \bar v_{zz}\geq -\frac{1000}{\bar{v}^5}.
\end{equation}
\end{theorem}

\begin{proof}
Arguing as in the proof of Proposition \ref{lem:rough.derivative.esti} (asymptotics of $h$), but now using the cylindrical estimates from \cite[Lemma 5.6]{CHH_translators} in lieu of the ones from \cite[Proposition 2.15]{ADS2}, given any $\eps>0$ we can find $\kappa>0$ and $\tau_\ast>-\infty$ with the following significance. If $M$ is $\kappa$-quadratic at some time $\tau_0\leq\tau_\ast$, then for all $\tau\leq\tau_0$ we have
\begin{equation}
    \sup_{q(z,\tau)\geq 2}  |h(z,\tau)-8|+\sup_{ q(z,\tau)\leq |\tau|^{-2/3}}h(z,\tau)\leq \eps.
\end{equation}
Next, using again the evolution equation from \cite[Proposition 5.3]{CHH_translators} we infer that 
 \begin{multline}
h_\tau= \frac{4qh_{zz}-20 q_zh_z}{4q|\tau|+q_z^2}+\frac{20qq_z^2h-8h^2+6q^2q_z^4}{q^2(4q|\tau|+q_z^2)}-\frac{z}{2} \left(1+\frac{1}{|\tau|}\right)h_z+\left(2-\frac{4}{q}-\frac{1}{|\tau|}\right)h+2q_z^2\\
+\frac{4(2h-qq_z^2)}{q(4q|\tau|+q_z^2)^2}\left[3q_zh_z-3q_z^4+(h+qq_z^2)(2|\tau|+q_{zz})+\frac{4q^2q_z^2(2|\tau|+q_{zz})^2}{4q|\tau|+q_z^2}\right]+E,
\end{multline}
where the only difference with Proposition \ref{prop:w.evol} (evolution of $h$) is the error term $E$, which, assuming $\kappa$-quadraticity at time $\tau_0\leq\tau_\ast$, in the region $\{ |\tau|^{-2/3}< q<2\}$ satisfies
\begin{equation}
|E|\leq e^{\frac{99}{100}\tau}\qquad \mathrm{for}\,\, \tau\leq\tau_0.
\end{equation}
Hence, if $\sup h^\ast(\cdot,\bar{\tau}_0)\geq 1000$ at some $\bar{\tau}_0\leq\tau_0$, then at any interior maximum $\bar{z}_0$ we get
\begin{equation}
    h^\ast_\tau(\bar{z}_0,\bar{\tau}_0) \leq -\frac{h^\ast(\bar{z}_0,\bar{\tau}_0)^2}{160 |\bar{\tau}_0|}.
\end{equation}
Utilizing this differential inequality as above, the assertion of the theorem follows.
\end{proof}

 \bigskip
 
 \section{Hessian estimate in tip region}
 
In this section, working in the unrescaled variables for convenience, we consider the function
\begin{equation}
H=Q_{xx},\qquad\mathrm{where}\,\, Q=V^p.
\end{equation}

\begin{proposition}[$H$ away from collar region]\label{lem:Hessian.soliton}
There exist $p<\infty$ such that for all $t\ll 0$ we have
\begin{equation}
 H(x,t)>0 \quad\mathrm
{if}\,\,\, V(x,t)=\theta\sqrt{|t|} \,\,\,\mathrm{or}\,\, \, V(x,t)\leq L\sqrt{|t|/\log|t|}.
\end{equation}
\end{proposition}

\begin{proof} Consider the scale invariant quantity
\begin{equation}
p^{-1}Q^{-1+\frac{2}{p}}H= VV_{xx}+(p-1)V_x^2.
\end{equation}
If $V=\theta\sqrt{|t|}$, namely if $\bar{v}=\theta$, then applying Theorem \ref{thm:Hessian.cyl.esti} (Hessian estimate in intermediate region) and Theorem \ref{thm:min.gradient.improved} (gradient estimate), provided that $p\geq 2000\theta^{-2}$, we infer that
\begin{equation}
\bar{v}\bar{v}_{zz}+(p-1)\bar{v}_z^2\geq -1000 \theta^{-4}+(p-1)\theta^{-2} > 0.
\end{equation}
On the other hand, observe that in the soliton region $\bar{v}\leq L/\sqrt{|\tau|}$ we have
\begin{equation}
   \bar{v}\bar{v}_{zz}+(p-1)\bar{v}_z^2=\frac{(p-1)Z_\rho(\rho)-\rho Z_{\rho\rho}(\rho)}{Z_\rho(\rho)^3}|\tau|  \qquad \mathrm{where}\quad \rho = |\tau|^{1/2}\bar{v}(z,\tau).
\end{equation}
Together with the tip asymptotics from Theorem \ref{thm_ADS1} (unique asymptotics) this implies the assertion.
\end{proof}

\begin{proposition}[{evolution of $H$}]\label{prop_evol_H}
The function $H=(V^p)_{xx}$ evolves by
\begin{align}
H_{t}=\frac{Q^{2-\frac{2}{p}}}{Q^{2-\frac{2}{p}}+\frac{1}{p^2} Q_x^2}H_{xx}+DH_{x}-(p-2)Q^{-\frac{2}{p}}H
+2(1-\tfrac{2}{p})Q^{-1-\frac{2}{p}}Q_x^2+R,
\end{align}
where the functions $D$ and $R$ are specified below.
\end{proposition}

\begin{proof}Recall from \cite[Section 1]{ADS1} that the unrenormalized profile function evolves by
\begin{equation}
    V_t=\frac{V_{xx}}{1+V_x^2}-\frac{1}{V}.
\end{equation}
Consequently, the function $Q=V^p$ evolves by
\begin{equation}
    Q_t=\frac{ Q^{2-\frac{2}{p}}Q_{xx}-(1-\frac{1}{p})Q^{1-\frac{2}{p}}Q_x^2}{ Q^{2-\frac{2}{p}}+\frac{1}{p^2} Q_x^2}-pQ^{1-\frac{2}{p}}.
\end{equation}
Differentiating this equation with respect to $x$ yields
\begin{multline}
        Q_{xt}=\frac{ Q^{2-\frac{2}{p}}Q_{xxx}-(1-\frac{1}{p})(1-\frac{2}{p})Q^{-\frac{2}{p}}Q_x^3}{ Q^{2-\frac{2}{p}}+\frac{1}{p^2} Q_x^2}\\
        -\frac{2Q^{1-\frac{2}{p}}Q_x[QQ_{xx}-(1-\frac{1}{p})Q_x^2][(1-\frac{1}{p})Q^{1-\frac{2}{p}}+\frac{1}{p^2}Q_{xx}]}{ (Q^{2-\frac{2}{p}}+\frac{1}{p^2} Q_x^2)^2} -(p-2)Q^{-\frac{2}{p}}Q_x.
\end{multline}
Differentiating again, and letting $D$ be the coefficient function of $Q_{xxx}$, namely
\begin{align}
    D&=\left(\frac{  Q^{2-\frac{2}{p}} }{ Q^{2-\frac{2}{p}}+\frac{1}{p^2} Q_x^2}\right)_x-\frac{2Q^{1-\frac{2}{p}}Q_x[\frac{2}{p^2}QQ_{xx}-\frac{1}{p^2} (1-\frac{1}{p})Q_x^2+(1-\frac{1}{p})Q^{2-\frac{2}{p}}]}{ (Q^{2-\frac{2}{p}}+\frac{1}{p^2} Q_x^2)^2},
\end{align}
this yields the claimed evolution equation, provided we set
\begin{align}
    R=& (1-\tfrac{1}{p})(1-\tfrac{2}{p})\left(\frac{  Q^{-\frac{2}{p}}Q_x^3}{ Q^{2-\frac{2}{p}}+\frac{1}{p^2} Q_x^2}\right)_x\nonumber\\
    & -\left(\frac{2Q^{1-\frac{2}{p}}Q_x}{ (Q^{2-\frac{2}{p}}+\frac{1}{p^2} Q_x^2)^2}\right)_x [QQ_{xx}-(1-\tfrac{1}{p})Q_x^2][(1-\tfrac{1}{p})Q^{1-\frac{2}{p}}+\tfrac{1}{p^2}Q_{xx}]\\
    &+\frac{2(1-\frac{2}{p})Q^{1-\frac{2}{p}}Q_x^2}{ (Q^{2-\frac{2}{p}}+\frac{1}{p^2} Q_x^2)^2} \left(Q_{xx}[(1-\tfrac{1}{p})Q^{1-\frac{2}{p}}+\tfrac{1}{p^2}Q_{xx}] -(1-\tfrac{1}{p})Q^{-\frac{2}{p}} [QQ_{xx}-(1-\tfrac{1}{p})Q_x^2]\right).\nonumber
\end{align}
This proves the proposition.
\end{proof}

\begin{lemma}[{differential inequality for $H$}]\label{lem_diff_ineq_H}
If $\inf_{V(x,t)\leq \theta} H(x,t_0)<0$ at some $t_0\ll 0$, then the infimum is attained at some $x_0$ satisfying $L\sqrt{|t_0|/\log|t_0|} < V(x_0,t_0) < \theta \sqrt{|t_0|}$ and we have
\begin{equation}
H_t(x_0,t_0) \geq -p\frac{H(x_0,t_0)}{|t_0|}.
\end{equation}
\end{lemma}

\begin{proof}The infimum is indeed attained in the interior of the collar region thanks to Proposition \ref{lem:Hessian.soliton} ($H$ away from collar region).
Next, by standard cylindrical estimates \cite[Proposition 2.15]{ADS2}, given any $\eps>0$, for $L<\infty$ sufficiently large and $\tau$ sufficiently negative in the region $\bar{v}(z,\tau)\geq L|\tau|^{-1/2}$ we have
\begin{align}\label{cyl_est_unresc}
     |Q_x| \leq C \varepsilon Q^{1-\frac{1}{p}} \qquad\mathrm{and}\qquad |Q_{xx}|\leq C \varepsilon Q^{1-\frac{2}{p}}.
\end{align}
Therefore, for $\tau$ sufficiently negative in the region $\bar{v}(z,\tau)\geq L|\tau|^{-1/2}$ we obtain
\begin{equation}
    |R|\leq C\eps \left( Q^{-\frac{2}{p}}|Q_{xx}|+Q^{-1-\frac{2}{p}}Q_x^2\right).
\end{equation}
Together with Proposition \ref{prop_evol_H} (evolution of H) and the second derivative test we conclude that
\begin{align}
H_{t}(x_0,t_0)\geq -(p-3)Q(x_0,t_0)^{-\frac{2}{p}}H(x_0,t_0).
\end{align}
This implies the assertion.
\end{proof}

\begin{theorem}[Hessian estimate in tip region]\label{thm_hess_tip}
For all $t\ll 0$ we have
\begin{equation}
VV_{xx}+(p-1)V_x^2\geq 0\qquad \textrm{if}\,\,  V(x,t)\leq \theta |t|^{1/2}.
\end{equation}
\end{theorem}

\begin{proof}
Set $\underline{H}(t):=\inf_{V(x,t)\leq \theta} H(x,t)$. If we had $\underline{H}(t_0)<0$ at some $t_0\ll 0$, then by Lemma \ref{lem_diff_ineq_H} (differential inequality for $H$) for all $t\leq t_0$ we would obtain
\begin{equation}
|t|^{-p}\underline{H}(t)\leq |t_0|^{-p}\underline{H}(t_0),
\end{equation}
contradicting \eqref{cyl_est_unresc}. This proves the theorem.
\end{proof}

Observer that Theorem \ref{thm:hessian.improved_intro} (Hessian estimate) now follows by combining Theorem \ref{thm:Hessian.cyl.esti} (Hessian estimate in intermediate region) and Theorem \ref{thm_hess_tip} (Hessian estimate in tip region).\\

Finally, let us establish the corresponding variant for noncollapsed translators in $\mathbb{R}^4$:

\begin{theorem}[Hessian estimate in tip region for translators]\label{thm_hess_tip_trans}
There exist $\kappa>0$, $\tau_\ast>-\infty$, $\theta>0$ and $p<\infty$, such that if $M$ is $\kappa$-quadratic at some time $\tau_0\leq\tau_\ast$ then for all $\tau\leq\tau_0$ we have
\begin{equation}
VV_{xx}+(p-1)V_x^2\geq 0\qquad \textrm{if}\,\,  V(x,t)\leq \theta |t|^{1/2}.
\end{equation}
\end{theorem}

\begin{proof}
Arguing as in the proof of Proposition \ref{lem:Hessian.soliton} ($H$ away from collar region), but now using Theorem \ref{thm:min.gradient.improved_intro_trans} (gradient estimate for translators) and Theorem \ref{thm:Hessian.cyl.esti_trans} (Hessian estimate in intermediate region for translators) in lieu of Theorem \ref{thm:min.gradient.improved} (gradient estimate) and Theorem \ref{thm:Hessian.cyl.esti} (Hessian estimate in intermediate region) we can find $\kappa>0$, $\tau_\ast>-\infty$, $\theta>0$, $L<\infty$ and $p<\infty$ with the following significance. If $M$ is $\kappa$-quadratic at some time $\tau_0\leq\tau_\ast$, then for all $t\leq -e^{-\tau_0}$ the function $H=(V^p)_{xx}$ satisfies
\begin{equation}
 H(x,t)>0 \quad\mathrm
{if}\,\,\, V(x,t)=\theta\sqrt{|t|} \,\,\,\mathrm{or}\,\, \, V(x,t)\leq L\sqrt{|t|/\log|t|}.
\end{equation}
Next, using again the evolution equation from \cite[Proposition 5.3]{CHH_translators} we infer that 
\begin{align}
H_{t}=\frac{Q^{2-\frac{2}{p}}}{Q^{2-\frac{2}{p}}+\frac{1}{p^2} Q_x^2}H_{xx}+DH_{x}-(p-2)Q^{-\frac{2}{p}}H
+2(1-\tfrac{2}{p})Q^{-1-\frac{2}{p}}Q_x^2+R+E,
\end{align}
where the only difference with Proposition \ref{prop_evol_H} (evolution of $H$) is the error term $E$, which, assuming $\kappa$-quadraticity at time $\tau_0\leq\tau_\ast$, in the collar region $\{ L\sqrt{|t|/\log|t|}\leq V\leq \theta\sqrt{|t|} \}$ satisfies
\begin{equation}
|E|\leq |t|^{-\frac{99}{100}}Q^{1-\frac{4}{p}}\qquad \mathrm{for}\,\, t\leq -e^{-\tau_0}.
\end{equation}
Observe that the error term $E$ can be absorbed into the good term $Q^{-1-\frac{2}{p}}Q_x^2$. Hence, using the cylindrical estimates from \cite[Lemma 5.6]{CHH_translators} in lieu of the ones from \cite[Proposition 2.15]{ADS2}
 we still get the same differential inequality
\begin{align}
H_{t}(\bar{x}_0,\bar{t}_0)\geq -(p-3)Q(\bar{x}_0,\bar{t}_0)^{-\frac{2}{p}}H(\bar{x}_0,\bar{t}_0),
\end{align}
and arguing as above we can conclude the proof of the theorem.
\end{proof}

Observer that Theorem \ref{thm:hessian.improved_intro_trans} (Hessian estimate for translators) now follows by combining Theorem \ref{thm:Hessian.cyl.esti_trans} (Hessian estimate in intermediate region for translators) and Theorem \ref{thm_hess_tip_trans} (Hessian estimate in tip region for translators).
 
\bigskip

\bibliography{profile_estimates}

\bibliographystyle{alpha}

\vspace{10mm}

{\sc Kyeongsu Choi, School of Mathematics, Korea Institute for Advanced Study, 85 Hoegiro, Dongdaemun-gu, Seoul, 02455, South Korea}\\

{\sc Robert Haslhofer, Department of Mathematics, University of Toronto,  40 St George Street, Toronto, ON M5S 2E4, Canada}\\

{\sc Or Hershkovits, Institute of Mathematics, Hebrew University, Givat Ram, Jerusalem, 91904, Israel}\\

\emph{E-mail:} choiks@kias.re.kr, roberth@math.toronto.edu, or.hershkovits@mail.huji.ac.il

\end{document}